\documentclass[10pt]{amsart}

\usepackage{amssymb}
\usepackage[latin1]{inputenc}

\usepackage{amsmath,amsfonts,amsthm,amstext,amscd}
\usepackage{hyperref}
\usepackage[cmtip,all]{xy}
\def\rm#1{\mathrm{#1}}
\def\cal#1{\mathcal{#1}}
\def\bb#1{\mathbb{#1}}
\def\lie#1{\mathfrak{#1}}
\def\lr#1{\left\langle #1 \right\rangle}
\def\lrr#1{\langle #1 \rangle}

\newtheorem{theorem}{Theorem}[section]
\newtheorem{theoremA}{Theorem}
\newtheorem{lemma}[theorem]{Lemma}
\newtheorem{conj}{Conjecture}

 \newtheorem{cor}[theorem]{Corollary}
 \newtheorem{lem}[theorem]{Lemma}
 \newtheorem{prop}[theorem]{Proposition}

\theoremstyle{definition}
\newtheorem{definition}[theorem]{Definition}
\newtheorem{example}[theorem]{Example}

\newtheorem{case}{Case}

\theoremstyle{remark}
\newtheorem{remark}[theorem]{Remark}

\numberwithin{equation}{section}

\begin{document}
\title{On  Riemannian Foliations over Positively Curved Manifolds}
%\title{Non-existence of Riemannian Foliations on Positively Curved Manifolds}
\author{L. D. Speran\c ca}
\address{Departamento de Matem\'atica, UFPR \\
 Setor de Ci\^encias Exatas, Centro Polit\'ecnico \\
 Caixa Postal 019081,  CEP 81531-990 \\
 Curitiba, PR, Brazil}
\email{lsperanca@ufpr.br}


\date{}

\begin{abstract}We prove that, under reasonable conditions, odd co-dimension Riemannian foliations cannot occur in positively curved manifolds. 
  \end{abstract}

%%% ----------------------------------------------------------------------
\maketitle

\section{Introduction}

Given a compact Riemannian manifold $(M,g)$, a foliation $\cal F$ in $ (M,g)$ is said to be Riemannian if its leaves are locally equidistant. We prove

\begin{theoremA}\label{mainthm} Let $\cal F$  be an odd-codimensional Riemannian foliation with bounded holonomy on a compact manifold $M$. Then $M$ has a plane with nonpositive sectional curvature.
\end{theoremA}

We say that a foliation has \textit{bounded holonomy} if there is a constant that uniformly bounds the norms of all  holonomy fields with unit initial condition. This class contains principal and associated bundles with compact structure group together with all regular homogeneous foliations defined by proper group actions  (section \ref{sec:bounded}).  \\

A particular case of Theorem \ref{mainthm} is interesting in its own right: the vertical warping of a connection metric. Its proof also helps understanding Theorem \ref{mainthm}. Let $\cal F$  be a  Riemannian foliation in $ (M,g_0)$ with totally geodesic leaves. For any  basic function $\phi$, we consider a new metric
\begin{gather}
g_\phi(X+\xi,X+\xi)=g_0(X,X)+e^{2\phi}g_0(\xi,\xi),
\end{gather}
for $X$  horizontal and $\xi$ vertical. 
We prove
\begin{theoremA}\label{thm:0} If $M$ is compact and $g_\phi$ has positive sectional curvature, then $\cal F$ has a fat point. I.e., there is a point $p\in M$ where the image of the O'Neill tensor, $A_X^*\xi$, is non-zero for all non-zero horizontal $X$  and vertical $\xi$.
\end{theoremA}
Theorem \ref{thm:0} has two straightforward implications:
\begin{cor}\label{cor:01}
If $\cal F$ has odd codimension, $g_\phi$ has a nonpositively curved plane.
\end{cor}
\begin{cor}\label{cor:02}
If $g_\phi$ is positively curved, then the dimension of $M$ is smaller then twice the co-dimension of $\cal F$.
\end{cor}
The first corollary is a strictly weaker version of Theorem \ref{mainthm}. The second states that Wilhelm's conjecture holds for $g_\phi$.

Theorem \ref{mainthm} and Corollary \ref{cor:01} generalize the classical result of Berger on zeros of Killing fields (\cite[pg. 193]{p}). The proof follows along the same lines, but requires two different tools. The main tool is \textit{dual holonomy fields} (sections \ref{sec:posintro} and \ref{sec:dhol}), analogous to the virtual Jacobi fields in \cite{verdianiziller}. The second tool is an auxiliary space,  the \textit{groupoid of infinitesimal holonomy transformations} (section \ref{sec:holbun}).
%\vspace{0.2cm}

\subsection{Notation}\label{sec:not} We mostly use the notation of \cite{gw}. We follow the usual nomenclature in Riemannian foliations, calling vectors tangent to leaves  \textit{verticals} and vectors orthogonal to leaves \textit{horizontals}. We denote the set of vertical vectors at $p\in M$ by $\cal V_p$ and the horizontals as $\cal H_p$. They define the vector bundles $\cal V$ and $\cal H$, respectively. The upper indexes $^h$ and $^v$ denote orthogonal projection to $\cal H$ and $\cal V$, respectively. Holonomy fields are denoted by $\xi$ and $\eta$, dual holonomy fields by $\nu$. Vertical vectors are also denoted  by $\xi$ and $\nu$ whenever the context is free from ambiguity, or by $\xi_0$ and $\nu_0$ otherwise.  All inner products and covariant derivatives are in $M$. We only consider non-singular foliations, i.e., all leaves have the same dimension.

\subsection{Holonomy and Dual Holonomy Fields}\label{sec:posintro}

Given a horizontal curve $c$ on $M$, a vector field $\xi$  is called a \textit{holonomy field} if it is  vertical and satisfies
\begin{gather}\label{eq:holfield}
\nabla_{\dot c}\xi=-A^*_{\dot c}\xi-S_{\dot c}\xi
\end{gather}
(compare \cite{gg} or \cite[pg. 17]{gw}). Holonomy fields are natural generalizations of action fields in principal bundles:  if $\cal F$ is given  by a principal bundle, all holonomy fields are obtained by restricting the action fields to horizontal curves (Proposition \ref{prop:bound2}). In this case, if $M$ is compact, we can find an uniform bound for the norm of all holonomy fields with unit initial data. 

Holonomy fields are entirely defined by the horizontal distribution $\cal H$ and do not sense the metric along the leaves. 
%\footnote{Send one paragraph below?}

A \textit{dual holonomy field} is a vertical field $\nu$ that satisfies
\begin{gather}\label{eq:predual}
\nabla_{\dot c}\nu=-A^*_{\dot c}\nu+S_{\dot c}\nu.
\end{gather}

Dual holonomy fields do sense the metric along the leaves. They appear in an interesting way in problems of integration along horizontal directions (as in \cite[Theorem 1]{DS-A-flat}).

Although we ask both holonomy and dual holonomy fields to be vertical, a vector field satisfying $\nabla_{\dot c}^h\xi=-A^*_{\dot c}\xi$ is vertical as long as it is vertical at a point (Remark \ref{rem:oncurves}). 
In particular, fixed a curve, holonomy and  dual holonomy fields are in one-to-one correspondence to their vertical initial data.

In contrast to virtual Jacobi fields (whose norms can explode), dual holonomy fields are  well behaved and give an interesting expression for the sectional curvature (Propositions \ref{prop:dualhol} and \ref{prop:KnuX}).
On the other hand, both constructions are closely related: in \cite{verdianiziller}, the authors define virtual Jacobi fields based on a choice of  Lagrangian subspace of Jacobi fields with respect to a natural symplectic form. Such Lagrangian spaces coincide with the $(n-1)$-dimensional family of Jacobi fields  with self-adjoint Ricatti operator in \cite{wilkilng-dual} or \cite[pg. 45]{gw}. Our dual holonomy fields are constructed in the same fashion, but using an isotropic subspace instead of a Lagrangian one. \\

In opposition to action fields, both holonomy fields and dual holonomy fields are, in principle, defined only along curves. This leads to the second tool: we introduce the groupoid of \textit{infinitesimal holonomy transformations} as an auxiliary space, where we can construct global objects similar to action fields,  whose restrictions to curves give rise to all holonomy and dual holonomy fields. \\

The precise definition of bounded holonomy follows below. An example of Riemannian foliation is given by the fibers of a Riemannian submersion. In this case, an important object that arises from $\cal F$ is its \textit{holonomy group} (\cite[pg. 13]{gw}). For instance,  if this group is compact, the submersion enjoys interesting properties (\cite{pro-flats,ss-hol,tapp-bounded-sub,tapp-vol-grwoth}). The condition of bounded holonomy is slightly more general and make sense for foliations. In section \ref{sec:bounded}, we prove that  it is satisfied for the common classes of Riemannian foliations mentioned in the beginning.
 
\begin{definition}\label{def:bhol}
We say that a Riemannian foliation has bounded holonomy if there is a constant $L$ such that,  for every holonomy field $\xi$, $||\xi(1)||\leq L||\xi(0)||$.
\end{definition}

When $M$ is compact,  this condition only depends on the horizontal distribution and not on all the structures invoked. We can also observe that a Riemannian foliation has bounded holonomy if and only if its dual holonomy fields satisfies the bound in Definition \ref{def:bhol} (Lemma \ref{lem:bound}). This is how the boundedness of holonomy comes into the proof of Theorem \ref{mainthm}.
\vspace{0.2cm}

The rest of the paper is divided in five sections. In section \ref{sec:1} we prove Theorem \ref{thm:0}. In section \ref{sec:holbun}, we construct the groupoid of infinitesimal holonomy transformations and in section \ref{sec:dhol} the dual holonomy fields. In section \ref{sec:proof1} we prove the main theorem and in section \ref{sec:tech} we explore further the hypothesis and methods. The boundedness of holonomy is only used on section \ref{sec:proof1}.\\

\section{Proof of Theorem \ref{thm:0}}\label{sec:1}

In a vertical warping of a connection metric, the function $\phi$ plays the role of the norm of the Killing field in the proof of Berger result (\cite[pg. 193]{p}). We prove Theorem \ref{thm:0} using Gray-O'Neill's formula to compute the vertizontal curvatures of $M$ at a maximum of $\phi$.
\vspace{0.2cm}

Since holonomy fields  of the metrics $g_0$ and $g_\phi$ coincide (these two metrics have the same horizontal distribution), we can explicitly compute the $S$-tensor in the metric $g_\phi$.  If $\xi$ and $\eta$ are holonomy fields along a horizontal curve $c$, 
\[g_{\phi}({\xi(t),\eta(t)})=e^{2\phi}g_0({\xi(t),\eta(t)})=e^{2\phi}g_0({\xi(0),\eta(0)})\,,\]
since the $g_0$-inner product of two holonomy fields is constant. The  $S$-tensor associated to  $g_\phi$, $S^\phi$,  is determined by
\[2g_\phi({S^\phi_{\dot c}\xi,\eta}) =-\frac{d}{dt}g_\phi({\xi,\eta}) =-\frac{d}{dt}[e^{2\phi}]g_0({\xi,\eta})=2g_\phi({-d\phi(\dot c)\xi,\eta}).\]
In particular, at a minimum of $\phi$, $S^\phi\equiv0$ and the equation for the unreduced sectional curvature of $\dot c$ and $\xi$ (see \cite[pg. 28]{gw} or \ref{eq:noK}) becomes
\begin{gather}\label{eq:Kphi}
 K(\xi,\dot{c})=-\frac{1}{2}||\xi||^2_0\rm{Hess}\,\phi(\dot c,\dot c)+||A^*_{\dot c}\xi||^2_\phi~.
\end{gather}
In addition, $\rm{Hess}\,\phi$ is nonnegative at this point. Therefore, if $K$ is positive, $||A^*_{\dot c}\xi||_\phi^2$ must be  nonzero for every  nonzero pair $\dot{c},\xi$. \\

Corollaries \ref{cor:01} and \ref{cor:02} are direct consequences of  linear algebra. For instance, if $\xi$ is fixed, the map $X\mapsto A_X^*\xi$ is skew-symmetric, consequently, it has  nontrivial kernel when $\rm{dim}~\cal H$ is odd. 

In general, the Hessian that appears in \eqref{eq:Kphi} is replaced by $\lr{(\nabla_XS)_X\xi,\xi}$, which is more difficult to control. The aim of the next sections  is to construct a function that plays the role of $\phi$ in a general Riemannian foliation. We introduce this function in section \ref{sec:proof1}. It is not defined on $M$ but in a subset of the groupoid of infinitesimal holonomy transformations.
\vspace{0.2cm}

Theorems \ref{mainthm}, \ref{thm:0} and \ref{thm:AS} suggest that  positively curved foliations may carry \textit{fat }objects. We state a conjecture along these lines:

\begin{conj}[Strong Wilhelm's Conjecture] Let $\cal F$ in $ M$ be a Riemannian foliation on a compact manifold with positive curvature. Then, it has a horizontal vector $X$ such that $A^*_X$ is injective.
\end{conj}

\section{The Infinitesimal Holonomy Bundle}\label{sec:holbun}

Here we present an auxiliary space used in the proof of Theorem \ref{mainthm}. The author believes that the language of groupoids is a natural language to introduce this object (see sections \ref{sec:Esmooth} and \ref{sec:AS}).

Foliations usually do not provide holonomy diffeomorphisms between leaves (as in the case of Riemannian submersions - \cite[pg. 12]{gw}). However, infinitesimal data can be recovered from holonomy fields.
For a horizontal curve $c:[0,1]\to M$, we define $h:\cal V_{c(0)}\to \cal V_{c(1)}$ as the linear isomorphism given by $h(\xi_0)=\xi(1)$, where $\xi(t)$ is the holonomy field along $c$ with initial condition $\xi(0)=\xi_0$. We call  $h$ an \textit{infinitesimal holonomy transformation}.

Another way to recover the infinitesimal data is to use local horizontal lifts. Let $\cal F$  be a Riemannian foliation on a complete Riemannian manifold $M$. Recalling that $\cal F$ is locally a Riemannian submersion, for every horizontal curve $c$, we can find a neighborhood $U$ of $c(0)$ inside the leaf and a map $\psi:U\times[0,1]\to M$ such that:
\begin{enumerate}
\item $\psi(x,0)=x$;
\item the map $x\mapsto \frac{d}{dt}\psi(x,t)$ defines a basic vector field for every $t$.
\end{enumerate}
We can always consider $d\psi_{(c(0),1)}:\cal V_{c(0)}\to \cal V_{c(1)}$  as an infinitesimal holonomy transformation, independent of the choice of $U$. It is easy to see that the set of such holonomy transformations  coincide with the set of holonomy transformations defined in the first paragraph.

Let $\cal E$ be the collection of all infinitesimal holonomy transformations defined by $\cal F$. $\cal E$ is naturally included in 
\begin{gather}
\rm{Aut}(\cal V)=\{h:\cal V_p\to\cal V_q~|~p,q\in M, h\in\rm{Iso}(\cal V_p,\cal V_q)\} ,
\end{gather}
where $\rm{Iso}(\cal V_p,\cal V_q)$ stands for the set of linear isomorphisms between $\cal V_p$ and $\cal V_q$. The natural operations on $\rm{Aut}(\cal V)$ defines a groupoid structure, the source and target maps  being defined on $h:\cal V_p\to \cal V_q$ by $\sigma(h)=p$ and $\tau(h)=q$, respectively. Moreover, $\cal E$ is closed by  composition and inversion in $\rm{Aut}(\cal V)$: if $h:\cal V_p\to \cal V_q$ is realized by the horizontal curve $c$ and $h':\cal V_q\to \cal V_r$ is realized by $c'$, then $h'\circ h$ is realized by the concatenation of $c$ and $c'$; $h^{-1}$ is realized by the curve $\tilde c$ defined by $\tilde c(t)=c(1-t)$.
The identity section $p\mapsto \rm{id}_{\cal V_p}$ is realized by constant curves.

We endow $\rm{Aut}(\cal V)$ with the topology defined by the submersion $\sigma\times \tau:\rm{Aut}(\cal V)\to M\times M$ along with the operator norm on $\rm{Iso}(\cal V_p,\cal V_q)$ induced by the metric on $M$.  The space $\cal E$  inherits a topology and a groupoid structure from $\rm{Aut}(\cal V)$.  Smoothness and related questions are discussed in section \ref{sec:tech}, but neither topology nor differentiability  will be used in the proof of Theorem \ref{mainthm}.
\vspace{0.2cm}

As  usual for groupoids, the restriction of $\tau$ to an \textit{orbit} of $\cal E$ defines a principal bundle: let $p\in M$ and denote $\cal E_p=\sigma^{-1}(p)\cap\cal E$. Then $\tau_p=\tau|_{\cal E_p}$ defines a principal bundle over $\tau(\sigma^{-1}(p))\cap\cal E=L^\#_p$, the dual leaf through $p$ (\cite{wilkilng-dual} or \cite[pg. 40]{gw}). The structure group of $\tau_p$, which we denote by $H_p$, is the set of infinitesimal holonomy transformations realized by closed horizontal loops based at $p$.

\begin{remark}\label{rem:oncurves}
In contrast to \cite{wilkilng-dual}, we consider holonomy fields along general horizontal curves instead of  broken geodesics. We prefer to deal with the generality of smooth curves since we believe that it expresses better the non-Riemannian nature of (infinitesimal) holonomy transformations. With smooth curves, we still recover the set of infinitesimal holonomy transformations defined by broken geodesics: every broken geodesic can be made a smooth curve by a reparametrization that makes its velocity flat  at cusps - from the second paragraph of this section, we see that infinitesimal holonomy transformations do not depend on the parametrization of their realizing curves.

To further justify how we are considering general curves, we observe that equation \eqref{eq:holfield} together with a vertical initial data defines a vertical vector field.
As mentioned, this fact is true for any field satisfying  $\nabla^h_{\dot c} \xi=-A^*_{\dot c}\xi$  (in particular, dual holonomy fields).

To prove that $\nabla^h_{\dot c} \xi=-A^*_{\dot c}\xi$ preserves verticality, observe that  $\nabla_{\dot c}^vZ=A_{\dot c}Z$ for every horizontal field $Z$, therefore,
\begin{align*}
\frac{d}{dt}\lr{\xi,Z}=\lr{\xi, \nabla_{\dot c}Z}+\lr{\nabla_{\dot c}^h\xi,Z}
=\lr{\xi, A_{\dot c}Z}+\lrr{\xi, \nabla^h_{\dot c}Z}-\lr{A^*_{\dot c}\xi,Z}=\lrr{\xi, \nabla^h_{\dot c}Z}.
\end{align*}
On the other hand, one can always obtain a horizontal frame of vector fields  satisfying $\nabla_{\dot c}^hZ=0$ (for instance, recall that $\cal F$ in $ M$ is locally a Riemannian submersion and use horizontal lifts of parallel vector fields.)
\end{remark}

\subsection{Examples}\label{rem:examples}Let us give a brief idea of  $\cal E_p$  in some cases:
	
\begin{case}\label{example1}
If $\cal F$ is given by the fibers of a Riemannian submersion $\bar \pi: M\to B$, its holonomy group  at $\bar \pi(p)$ acts via  diffeomorphisms on the fiber  $F=\bar \pi^{-1}(\bar \pi(p))$ (see \cite[pg. 13]{gw}). In this case, $H_p$ coincides  with the image of the isotropy representation of the holonomy group at $\cal V_p=T_pF$.
\end{case}
\begin{case}\label{example2}
If $\bar{\pi}:M=P\stackrel{\bar \pi}{\to} B$ is a principal $G$-bundle, holonomy fields are restrictions of action fields to horizontal curves (see Lemma \ref{lem:Killingfol} for a proof), i.e.,  if  $\hat \xi$ is an action field and $c$ is a horizontal curve, $\xi(t)= \hat{\xi}(c(t))$ is a holonomy field along $c$. In this case, $H_p$ is trivial since
\begin{gather}\label{eq:principalhol}
\xi(0)=\hat \xi(c(0))=\hat \xi(c(1))=\xi(1)
\end{gather}
for every closed horizontal loop $c$. Furthermore, $\bar{\pi}\circ \tau_p:\cal E_p\to B$ is isomorphic to the bundle reduction defined by the connection on $P$ (denoted by $P(p)$ in \cite[II.7, Theorem 7.1]{knI}),  which coincides with $L^\#_p$.
\end{case}
\begin{case}\label{example3}
More generally, for $\bar\pi:M\to B$ a Riemannian submersion with totally geodesic fibers, $\bar \pi\circ \tau_p:\cal E_p\to B$ is isomorphic to the \textit{holonomy bundle} of the restriction $\bar \pi_p=\bar \pi|_{L^\#_p}:L_p^\#\to B$: in general, we can define a principal bundle $\bar P$ by gathering all holonomy diffeomorphisms with domain $F$.  When the submersion has totally geodesic fibers, all elements of $\bar P$ are inside the bundle defined by all isometries from $F$ to any fiber (the last bundle is denoted by $P$ in \cite[Theorem 2.7.2]{gw}). Moreover, the last bundle inherits a natural connection from $\bar \pi$ and $\bar P$ coincides with its connection reduction passing through the identity. The structure group of $\bar P$ is the holonomy group of $\bar \pi$ at $\bar \pi(p)$, which we denote here by $G$.

To see how $\cal E_p$ realizes the holonomy bundle, assume that $M=L^\#_p$ and observe that $G$ acts naturally on $\cal E_p$: if $g\in G$ and $h\in \cal E_p$ are realized by the curves $\beta$ and $c$ respectively, we set $h\cdot g$ as the transformation realized by the concatenation of $\beta$ followed by the lif of $\bar{\pi}\circ c$ to $\beta(p)$. Equivalently, $h\cdot g=\tilde h\circ dg_p$, where $\tilde h$ is the infinitesimal holonomy transformation defined by the lift of $\bar \pi\circ c$ to $g(p)$. This action is free, since $g$ is an isometry: $\tilde h\circ dg_p=h$ if and only if $g(p)=p$ and $dg_p=\rm{id}$, therefore  $g=\rm{id}$.
\end{case}

\begin{case}\label{example4}
The arguments in Case \ref{example3} remain valid whenever the holonomy group of the submersion is compact (see Theorem \ref{thm:2}). 
%In which case we also conclude that the twisted product $\cal E_p{\times_{H_p}} F$ together with the projection $(h,q)\mapsto \bar \pi_p(h)$  is isomorphic to $M\to B$.
\end{case}

These facts are the main motivations for defining $\cal E$ and $\cal E_p$.

\subsection{The Natural Action and Horizontal Lifts}\label{sec:horlift}The groupoid $\cal E$  acts naturally on vertical vectors. As we shall see, this action gives rise to all holonomy and dual holonomy fields. Let $\pi:\cal V\to M$ be the bundle of vertical vectors and define
\begin{align}\label{eq:action}
\zeta:\cal E {_\sigma\times_\pi} \cal V&\to \cal V\\\nonumber
(h,\xi)&\mapsto h(\xi)\,,
\end{align}
where $\cal E {_\sigma\times_\pi}\cal V$ 
is the usual fibered product:
\[\cal E {_\sigma\times_\pi} \cal V=\{(h,\xi)\in \cal E\times \cal V~|~ \sigma(h)=\pi(\xi)\}.\]

We easily observe that the restriction $\zeta|_{\cal E_p\times \cal V_p}:\cal E_p\times \cal V_p\to \cal V$ defines $\cal V$ as a linear bundle associated to $\cal E_p$.

Holonomy fields give  a natural way to lift horizontal curves from $M$ to $\cal E$. For any horizontal curve $c:[0,1]\to M$, we define $\hat c:I\to \cal E$ as
\[\hat c(t)\xi_0= \xi(t),\]
where $\xi$ is the holonomy field along $c$ with initial condition $\xi(0)=\xi_0$. Observe that $\sigma(\hat c(t))=c(0)$ and $\tau(\hat c(t))=c(t)$.

Given $h\in \cal E_p$ and a horizontal curve $c$ starting at $\tau(h)$, we define the \textit{$\tau_p$-horizontal lift of $c$ at $h$} as $\hat c_h(t)=\hat c(t)h$. It follows immediately from the definitions that all holonomy fields on $L_p^\#$ are recovered by these lifts. We state it as a proposition.

\begin{prop}\label{prop:hol}
Let $h\in \cal E_p$ and $c$ be a horizontal curve with $c(0)=\tau(h)$.  Then
\begin{enumerate}
\item Given $\xi_0\in \cal V_p$,  $\xi(t)=\zeta(\hat c_h(t),\xi_0)$ is a holonomy field along $c$;
\item Given a holonomy field $\xi$ along $c$, then $\xi(t)=\zeta(\hat c_h(t),h^{-1}(\xi(0))$.
\end{enumerate}
\end{prop}

\subsection{Bounded Holonomy}\label{sec:bounded}
With Proposition \ref{prop:hol}  at hand, we can give sufficient conditions for bounded holonomy. 
%We intend to present this condition as a good replacement for compact holonomy in submersions (\cite{tapp-bounded-sub}). 
We begin with an alternative characterization of bounded holonomy.

\begin{lem}\label{lem:opbound}
$\cal F$ has bounded holonomy if and only there is a constant $L$ that bounds the operator norm of all elements in $\cal E$. That is, if $h\in\cal E$, $||h||\leq L$.
\end{lem}

%\begin{lem}
%$\cal F$ in $ M$ has bounded holonomy if and only if there is an uniform bound $L$ for the operator norms of all elements in $\cal E$.
%\end{lem}

Moreover, the operator norm is continuous in the topology induced by $\rm{Aut}(\cal V)$, since it is continuous in $\rm{Aut}(\cal V)$. 
%Let us proceed to the first class of foliations with bounded holonomy.

\begin{prop}\label{prop:bound1}
Let $M$ be compact and $\cal F$  be the Riemannian foliation given by the fibers of a Riemannian submersion $\bar \pi:M\to B$ with compact structure group. Then $\cal F$  has bounded holonomy.
\end{prop}
\begin{proof}
Since $M$ is compact, every point in $B$ can be connected to a given point $x$ by a geodesic whose length is less then the diameter of $M$. Denote the holonomy diffeomorphism defined by a length minimizing geodesic that connects $x$ to $y$ by $\psi^y$ (the choice of the geodesic is irrelevant). Since the length of the geodesics are uniformly bounded, so are $d\psi^y$ (\cite[Proposition 2.2]{tapp-bounded-sub}). The same is true for the differential of any element in the holonomy group, since the last is contained in a compact structure group. That is,  the differentials  $d\psi$ are bounded for all $\psi$ in the holonomy group at $x$.

The proposition follows from the fact that every infinitesimal holonomy transformation can be decomposed as $h=(d\psi^y)^{-1}d\psi d\psi^{y'}$ for some $\psi$ in the holonomy group, $y=\pi(\tau(h))$ and $y'=\sigma(h)$.
\end{proof}

When the foliation is given by a family of Killing fields, we prove that all holonomy fields are the restriction of Killing fields in the family. This done, the desired bound is given in  terms of the norms of the elements in the family.

\begin{lem}\label{lem:Killingfol}
Let $\cal F$  be a Riemannian foliation defined by the action of a Lie algebra $\lie m$ of Killing vector fields. Then, for every $\hat{\xi}\in \lie m$ and every horizontal curve $c$, $\xi(t)=\hat{\xi}(c(t))$ is a holonomy field.
\end{lem}
\begin{proof}
Let $\Xi_\theta$ be the flow of $\hat \xi\in \lie m$. $\cal V$ is preserved by $d\Xi_\theta$ since $\cal V$ is spanned by  the Lie algebra  $\lie m$. Since $\hat{\xi}$ is Killing, $d\Xi_\theta$ also preserves the horizontal distribution. Therefore, if $c$ is a horizontal curve starting at $p$, $\psi(\theta,t)=\Xi_\theta(c(t))$ is a collection of horizontal curves. Furthermore,  $X=\frac{\partial \psi}{\partial t}$ is a horizontal vector field along the image of $\psi$ which satisfies $[\hat \xi,X]=0$. In particular,  $\nabla^v_X\xi=\nabla^v_{\xi}X=-S_X \xi$, which   proves that  $\xi(t)$ satisfies equation \eqref{eq:holfield} ($\nabla^h_X\xi=-A^*_X\xi$ since $\xi$ is vertical.)
\end{proof}

As a corollary of Lemma \ref{lem:Killingfol}, we get $H_p=\{\rm{id}\}$. If $M$ has only one dual leaf, we see that $\cal E=M\times M$, which is a compact subset of $\rm{Aut}(\cal V)$ if $M$ is compact (thus bounding the operator norms of elements of $\cal E$). A proof for the general case is given below.

\begin{prop}\label{prop:bound2}
Let $M$ be compact and $\cal F$  be a  Riemannian foliation defined by a Lie algebra $\lie m$ of Killing fields. Then $\cal F$  has bounded holonomy.
\end{prop}
\begin{proof}
Given $p\in M$, denote by $e_p:\lie m\to \cal V_p$ the evaluation map and fix an inner product on $\lie m$. Since $\lie m$ spans the vertical space at every point, the space 
\[E=\{(p,\hat{\xi})\in M\times \lie m~|~\hat{\xi}\in (\ker e_p)^{\bot} \}\]
defines a smooth vector bundle over $M$. We further endow it with the metric induced by $e_p$.  Denote the sphere bundle of $E$ by $S(E)$. We can define a function $r:S(E)\to \bb R$ as $r(p,\hat \xi)=\max_{q\in M}||\hat \xi(q)||$. Since $r$ is continuous and $S(E)$ compact, it has a maximum $L$. Lemma \ref{lem:Killingfol} guarantees that this is the desired bound.
\end{proof}

\section{Dual Holonomy Fields}\label{sec:dhol}

When leaves are totally geodesic manifolds,  the inner product between two holonomy fields is always constant.  In this case, one can identify holonomy fields as their own duals. In any other situation, dual holonomy fields are introduced to play this role. We give three equivalent characterizations of these objects.
\begin{prop}\label{prop:dualhol}
Let $\nu$ be a vertical field along a horizontal curve $ c$ on $M$. Then, the following conditions are equivalent:
\begin{enumerate}
\item For any holonomy field $\xi$, $\lr{\xi(t),\nu(t)}$ is constant;
\item If $\hat c_{h}$ is a $\tau_p$-horizontal lift, then $\nu(t)=\zeta((\hat {c}_h(t)^*)^{-1},h^{*}(\nu(0)))$;
\item $\nabla_{\dot c}\nu=-A^*_{\dot c}\nu+S_{\dot c}\nu$.
\end{enumerate}
We call a vertical field satisfying any of these conditions as a dual holonomy field.
\end{prop}
\begin{proof}
Items (1) and (2) are clearly equivalent since $\xi$ is a holonomy field if and only if $\xi(t)=\zeta(\hat{c}_h(t),h^{-1}(\xi(0)))$ (Proposition \ref{prop:hol}). To verify the equivalence between (1) and (3), note that, for a holonomy field $\xi$,
\begin{gather}
\frac{d}{dt}\lr{\xi,\nu}=\lr{\nabla^v_{\dot c}\xi,\nu}+\lr{\xi,\nabla^v_{\dot c}\nu}=\lr{\xi,\nabla^v_{\dot c}\nu-S_{\dot c}\nu},
\end{gather}
which is zero for all holonomy fields if and only if $\nabla^v_{\dot c}\nu=S_{\dot c}\nu$. 
Moreover, any vertical vector field satisfies  $\nabla_{\dot c}^h\nu=-A^*_{\dot c}\nu$.\end{proof}

Item (1) connects dual holonomy fields to the virtual Jacobi fields defined on \cite{verdianiziller}. According to item (2), if we define $\bar \zeta:\cal E {_\sigma\times_\pi} \cal V\to \cal V$ as 
\begin{gather}
\bar \zeta(h,\nu)=(h^*)^{-1}(\nu),
\end{gather}
then, in analogy to Proposition  \ref{prop:hol}, $\bar \zeta(\hat c_h(t),\nu)$ is a dual holonomy field along $c$ and any dual holonomy field can be expressed in this way. Item  (3) provides an useful expression for the sectional curvature of the plane spanned by $\dot c$ and $\nu$ (Proposition \ref{prop:KnuX}).

\subsection{The Curvature Equation}\label{sec:proofKnuX}
\begin{prop}\label{prop:KnuX}
Let $\nu$ be a dual holonomy field and $\gamma$ a horizontal geodesic. Then, the unreduced sectional curvature $K$ of  the pair $\dot \gamma, \nu$ along $\gamma$ is given by
\begin{gather}\label{eq:K}
K(\dot \gamma,\nu)=\frac{1}{2}\frac{d^2}{dt^2}||\nu||^2-3||S_{\dot \gamma}\nu||^2+||A^ *_{\dot \gamma}\nu||^2.
\end{gather} 
\end{prop}
\begin{proof}
 Recalling Grey-O'Neill's vertizontal curvature equation (\cite[pg. 28]{gw}), we have
\begin{gather}\label{eq:KK}
K(\dot \gamma,\nu)=\lr{(\nabla_{\dot \gamma}S)_{\dot \gamma}\nu,\nu}-||S_{\dot \gamma}\nu||^2+||A^*_{\dot \gamma}\nu||^2.
\end{gather}
Computing the first term in \eqref{eq:KK}, we get:
\begin{align*}
\lr{(\nabla_{\dot \gamma}S)_{\dot \gamma}\nu,\nu}&= \frac{d}{dt}\lr{S_{\dot \gamma}\nu,\nu}-\lr{S_{\dot \gamma}\nu,\nabla_{\dot \gamma}\nu}-\lr{S_{\dot \gamma}\nabla_{\dot \gamma}\nu,\nu}\\
&=\frac{d}{dt}\lr{\nabla_{\dot \gamma}\nu,\nu}-||S_{\dot \gamma}\nu||^2-\lr{\nabla_{\dot \gamma}^v\nu,S_{\dot \gamma}\nu}\\
&=\frac{1}{2}\frac{d^2}{dt^2}||\nu||^2-2||S_{\dot \gamma}\nu||^2.
\end{align*}

\end{proof}

The analogous equation for a holonomy field $\xi$ is
\begin{gather}\label{eq:noK}
K(\dot \gamma,\xi)=-\frac{1}{2}\frac{d^2}{dt^2}||\xi||^2+||S_{\dot \gamma}\xi||^2+||A^ *_{\dot \gamma}\xi||^2.
\end{gather} 
The advantage of \eqref{eq:K} is the minus sign in front of $||S_{\dot \gamma}\nu||^2$.

\section{Proof of  Theorem 1}\label{sec:proof1}

Fix an unitary $\nu_0\in \cal V_p$ and define a real function $\rho_{\nu_0}:\cal E_p\to \bb R$ as 
%\footnote{here you lost me already: one has to understand $\cal E_p$ in order to understand the proof and I am not an android. Reserve a morning or so to explain it to Lino	(or com back via USP and explain it to Alexandrino and Struchiner, answering your own final question as a byproduct), or perhaps Olivier here. }
\begin{gather}
\rho_{\nu_0}(h)=||\bar \zeta(h,\nu_0)||^2.
\end{gather}
We use this function to replace $\phi$ in the proof of Theorem \ref{mainthm}. For instance, at a maximum of $\rho_{\nu_0}$, equation \eqref{eq:KK} guarantees that $\lr{(\nabla_XS)_X\nu,\nu}$ is non-positive and we can use linear algebra to deal with the $A$-term.

\begin{theorem}\label{thm:max}
Suppose that $M$ is compact and $\cal F$ has bounded holonomy. Then, there exists a non-zero $\nu\in \cal V$ such that, for every $X\in \cal H_{\pi(\nu)}$, $K(X,\nu)\leq ||A^*_X\nu||^2$.
\end{theorem}
%
%\begin{theorem}\label{thm:max}
%Suppose that $M$ is compact and $\cal F$ has bounded holonomy. Then, there is a sequence $h_k\in\cal E_p$ such that, for any sequence $X_k\in \cal H_{t(h_k)}$ of unitary vectors, there is a subsequence $\{k_i\}$ and a constant $C$ satisfying
%\begin{gather}\label{eq:limit}
%sec_M(X_{k_i},\bar{\zeta}(h_{k_i},\nu_0))\leq C||A^*_{X_{k_i}}\bar\zeta(h_{k_i},\nu_0)||^2+\frac{1}{k_i}.
%\end{gather}
%\end{theorem}

Before proving Theorem \ref{thm:max}, we make a connection between the hypothesis and the map $\rho_{\nu_0}$.

\begin{lemma}\label{lem:bound}
A foliation $\cal F$ in $ M$ has bounded holonomy if and only if there are constants $\bar l, \bar L>0$ such that  $\bar l||\nu_0||\leq \rho_{\nu_0}(h)\leq \bar L||\nu_0||$, for all $(h,\nu_0)\in \cal E{_\sigma\times_\pi} \cal V$.
\end{lemma}
\begin{proof}
Let $L$ be a bound for the norm of holonomy fields with unit initial condition. According to Lemma \ref{lem:opbound}, $||h||\leq L$ for all $h\in\cal E$. On the other hand, since $\cal E$ is closed by inversion, $||(h^*)^{-1}||=||h^{-1}||\leq L$ for all $h\in\cal E$. But, according to (2) of Proposition \ref{prop:dualhol}, this is equivalent to have a bound on all dual holonomy fields with unit initial condition. The constants $\bar L$ and $\bar l$ can be taken as $L$ and $L^{-1}$, respectively.
\end{proof}

\begin{proof}[Proof of Theorem \ref{thm:max}:]
Given $\nu\in\cal V_q$ and $X\in\cal H_{q}$, we take advantage of equation \eqref{prop:KnuX} by exploring the function $f(t)=||\nu(t)||^2$, where $\nu(t)$ is the dual holonomy field defined by $\nu$ along the geodesic  spanned by $X$. From \ref{prop:KnuX}, 
\begin{gather}\label{eq:proof1}
\frac{1}{2}f''(0)=K(\dot \gamma,\nu)+3||S_{\dot \gamma}\nu||^2-||A^ *_{\dot \gamma}\nu||^2.
\end{gather}
In particular, we complete the proof by finding  $\nu$  that satisfies $f''(0)\leq 0$ for each $X\in\cal H_q$. Such a vector can be found by `maximizing' the function $\rho_{\nu_0}$: let $\{h_k\in\cal E_p\}$ be a sequence whose images, $\{\rho_{\nu_0}(h_k)\}$, converges to the supremum of $\rho_{\nu_0}$. We shall see that $\nu$ can be taken as any accumulation point for $\{\bar{\zeta}(h_k,\rho_0)\}$.

For simplicity, we assume for the sequence $\{h_k\}$ above, that $\{\bar \zeta(h_k,\nu_0)\}$ converges to some $\nu$. The limit, $\nu$, must be non-zero since there are constants $\bar l,\bar L>0$ such that $\bar l\leq||(h^*_k)^{-1}||\leq \bar L$ (lemma \ref{lem:bound}).

Observe that the proof can be concluded immediately if $\{h_k\}$ converges to some $h\in \cal E_p$. In this case, $\rho_{\nu_0}(h)$ is a maximum and, for every $X\in \cal H_{\tau(h)}$, the function $f(t)=||\nu(t)||^2$ has a maximum at $0$. The theorem follows by equation \eqref{eq:proof1}. In the general case, we show that we can approach the function $f$ by similar functions.

%
%is broken into two parts. We first construct a sequence of  functions $\{f_k:\bb R\to \bb R\}$ where 
%\begin{enumerate}
%\item the sequence of values $\{f_k(0)\}$ converges to a common supremum of the family $\{f_k\}$;
%\item the derivatives of $\{f_k\}$, up to a fixed order, can be uniformly bounded;
%\item the second derivatives of $f_k$ are related to the curvature of $M$.
%\end{enumerate}
%Then we prove that the second derivatives of the $f_k$'s at the origin converges to a non-positive number. This completes the proof. The first part comprehend the main ideas of the paper. The second part is technical and can be proved in general for any sequence satisfying (1) and (2). 

%Let $h_k$ be a sequence in $\cal E_p$ whose images $\rho_{\nu_0}(h_k)$ converge to a supremum of $\rho_{\nu_0}$. 

Fix $X\in \cal H_{\pi(\nu)}$ and let $\{X_k\in\cal H_{\tau(h_k)}\}$ be sequence of horizontal vectors converging to $X$. Consider the family of real functions $\{f_k\}$ defined by $f_k(t)=||\nu_k(t)||^2$, where $\nu_k(t)$ is  the  dual holonomy field defined by $h_k\nu_0$ along $\exp(tX_k)$. The sequence $\{f_k\}$ clearly converges pointwise to $f$.
To conclude that the convergence  is $\cal C^\infty$,  observe that the derivatives of  $f_k$ are expressed in terms of $S$ and the covariant derivatives of $S$. Since $M$ is compact, we can uniformly  bound any finite number of them. 

%We claim that these two facts imply the existence of a subsequence $\{f_{k_i}\}$ where $f_{k_i}''(0)$ converges to a non-positive number. This proved, Theorem \ref{thm:max} is settled by the following equation (compare section \ref{sec:proofKnuX}):
%\begin{gather*}
%f''_k(0)=\lr{(\nabla_{X_k}S)_{X_k}\bar{\zeta}(\nu_0,h_k),\bar{\zeta}(\nu_0,h_k)}+2||S_{X_k}\bar \zeta(\nu_0,h_k)||^2  
%\end{gather*}
%and the fact that we can take $||\bar{\zeta}(h_k,\nu_0)||^2=\rho_{\nu_0}(h_k)$ bounded away from zero (we can assume that the values of $\{f_{k_i}(0)\}$ are near enough to the supremum of $\rho_{\nu_0}$).
%\vspace{0.2cm}

Assume, by contradiction, that $f''(0)>2d$ for some $X\in\cal H$ and $d>0$. Let $k$ be big enough so that $f''_k(0)>d$. Then, the second order Taylor expansion of $f_k$ gives:
\begin{gather*}
f_{k}(\epsilon)=f_{k}(0)+f_{k}'(0)\epsilon+f''_{k}(0)\frac{\epsilon^2}{2}+O_{k}^{(2)}(\epsilon)\,.
\end{gather*}

The uniform bound on the derivatives  guarantees an uniform bound $|O^{(2)}_k(\epsilon)|<l\epsilon^3/2$ for all $k$. In fact, for each $k$, there is a $c_k\in \bb R$ such that
\[|O^{(2)}_k(\epsilon)|<\left|\frac{f'''_k(c_k)}{3!}\right||\epsilon|^3.\]
Now, taking $d/4l<\epsilon<d/2l$, we have
\begin{align*}
f_{k}(\epsilon)>f_{k}(0)+f'_{k}(0)\epsilon+\frac{\epsilon^2}{2}(d-l\epsilon)>f_{k}(0)+f'_{k}(0)\epsilon+\frac{d^3}{64l^2},
\end{align*}
which converges to a value strictly bigger then $f(0)$, unless $f'(0)$ is negative. However, following along the same lines, we conclude that
\begin{align*}
f_{k}(-\epsilon)>f_{k}(0)-f'_{k}(0)\epsilon+\frac{d^3}{64l^2}
\end{align*}
is strictly bigger then $f(0)$ for big $k$  if $f'(0)$ is negative. Observing that $f_k(t)=\rho_{\nu_0}(\hat c_{k}(t))$, where $\hat c_k$ is  the $\tau_p$-horizontal lift of  $\exp(tX_k)$ at $h_k$, we contradict the fact that $f(0)$ is a supremum for $\rho_{\nu_0}$.
\end{proof}

%A step further also proves that $f'_k(0)$ converges to zero (it is not used in the proof).
%\vspace{0.2cm}

As in the proof of Corollary \ref{cor:01}, the dimension hypothesis on Theorem \ref{mainthm} guarantees the existence of $X\in\cal H$ such that $||A^*_{X}\nu||^2$ vanishes, concluding the proof of Theorem \ref{mainthm}.

%\begin{lem}
%Let the rank of $\cal H$ be odd and $\nu\in \cal V_p$. Then,  there is a non-zero horizontal $X\in \cal H$ such that $A^*_X\nu=0$.
%\end{lem}
%\begin{proof}
%Since the linear map $X\mapsto A^*_X\bar\zeta(h,\nu_0)$ is skew-symmetric, it has kernel.
%\end{proof}
%
%The last lemma, combined with \eqref{eq:limit}, gives a sequence of vertizontal planes whose sectional curvatures goes to zero. Therefore, any convergent subsequence gives us a plane of zero sectional curvature, concluding the proof of Theorem \ref{mainthm}.

\section{Final Remarks}\label{sec:tech}

\subsection{Remarks on the proof of Theorem \ref{mainthm}}

We first observe that we proved a slightly better version of Theorem \ref{mainthm}:
\begin{theorem}
Let $\cal F$ be a Riemannian foliation on $M$ with positive vertizontal curvature and odd codimension. Then, for each non-zero $\nu_0\in\cal V_p$, the orbit $H_p \nu_0$ is unbounded. 
\end{theorem}

As a second remark, we observe that the use of the groupoid could be avoided: one could produce the maps $f_k$ in \ref{sec:proof1} considering a sequence of curves and dual holonomy fields that converges to a supremum.  The introduction of the groupoid allows to organize better these ideas and to present a proof along the same lines of Berger's result  or Theorem \ref{thm:0}. 
%Moreover, the space $\cal E$ may allow a better understanding of the situation.
%Our preference for the $\cal E_p$ approach is not only for organization sake, but for clarification of section \ref{sec:bounded} and possible contributions: the author believe that the Morse theory of $\rho_{\nu_0}$ could throw some light on Theorem \ref{mainthm} in the general case. Such an approach could be awkward without an auxiliary space. 

\subsection{Smoothnes of $\cal E$}\label{sec:Esmooth} We believe that Morse Theory on $\cal E$ could be a genuine approach to the generalization of Theorem \ref{mainthm} to unbounded holonomy. However, the groupoid $\cal E$ is not smooth in general. 

A pathology one may  encounter is an abrupt change on the dual leaves, as illustrated below.
\begin{example}\label{example:1}
Let $S^7\to S^4$ be the $S^3$-principal Hopf fibration. Consider the action of $S^3$ on itself by conjugation and form the associated bundle $\pi:S^7{\times_{S^3}}S^3\to S^4$ with it. Now, for any point $x\in S^7$, the subset $S^7{\times_{S^3}}\{1\}\subset S^7{\times_{S^3}}S^3$ is homeomorphic to $S^4$ and is a horizontal section of $\pi$. In particular, $H_{[x,1]}$ is trivial and $\sigma^{-1}([x,1])$ is homeomorphic to $S^4$. On the other hand, $\sigma^{-1}([x,i])$ is at least six dimensional. Therefore, there is no smooth structure on $\cal E$ that makes it connected and $s$ a smooth submersion, if so, it would have diffeomorphic fibers.\end{example}

This pathology does not occur in positive curvature (\cite[Theorem 1]{wilkilng-dual}). When the holonomy is bounded, we see (Theorem \ref{prop:AS}) that it does not occur even if we assume only positive vertizontal curvatures.  We also remark that generically $\cal E$ should coincide with $\rm{Aut}(\cal V)$, i.e., given a Riemannian foliation $\cal F$ on $(M,g)$, there is another metric $g'$, arbitrarily close to $g$, where $\cal F$ is a Riemannian foliation on $(M,g')$ and $\cal E$ coincides with $\rm{Aut}(\cal V)$.

The expected smooth structure on $\cal E$ should be as a submanifold of $\rm{Aut}(\cal V)$. In this sense, we say that a foliation has \textit{regular holonomy} if the inclusion $\cal E\subset \rm{Aut}(\cal V)$ defines a manifold structure on $\cal E$.

\subsection{A dual leaf theorem}\label{sec:AS}

We prove that the pathology of Example \ref{example:1} is not present in our context.

\begin{theorem}\label{prop:AS}
Let $\cal F$ be a  Riemannian foliation with bounded holonomy and positive  vertizontal curvature. Then $\cal F$ has only one dual leaf.
\end{theorem}

We prove this theorem by giving an Ambrose-Singer type of description for the tangent of the dual leaves. Recall the case of principal bundles, where the Ambrose-Singer theorem describes the tangent to the dual leaf via the curvature two-form $\Omega$ (see Case \ref{example2} in \ref{rem:examples} and \cite[II.7, Theorem 7.1]{knI}). It states that (after proper identification), the vertical part  of $TL^\#_p$ is spanned by $\{\Omega(X,Y)\}$, where $X,Y$ runs through all horizontal vectors on the bundle. The curvature two-form is not present in our context and its best replacement is Grey-O'Neill $A$-tensor, whose image resides on different fibers of $\cal V$ (making it impossible to get all values together). However, we can define the set
\[\cal A_p=span\{h^{-1}(A_XY)~|~X,Y\in\cal H_{\tau(h)},~h\in \cal E_p\}.\]
The inclusion $\cal A_p\subset T_pL^\#_p\cap \cal V_p$ is clear but it should be straightforward that $\cal A_p=T_pL^\#_p\cap \cal V_p$. Given a horizontal curve $c:\bb R\to M$, we can define a (usually) much smaller set: 
\[\cal C(c)=span\{\hat c(t)^{-1}(A_{\dot c(t)}Z)~|~Z\in \cal H_{c(t)}\text{ and }t\in\bb R\}\subset \cal V_{c(0)}.\]
Let $\xi$ be a Jacobi holonomy field whose initial value is orthogonal to the dual leaves. In the context of non-negatively curved manifolds, \cite{wilkilng-dual}  proves that such a holonomy field stays orthogonal to dual leaves. In general, dual holonomy fields exhibit a similar behavior with respect to $\cal C(c)$ (or $\cal A_p$). For convenience, denote by $c_s$ the curve $c_s(t)=c(s+t)$.

\begin{lem}\label{lem:Anu=0}
If $\nu_0\bot \cal C(c)$, then $\nu$, the dual holonomy field defined by $\nu_0$ along $c$, satisfies $\nu(s)\bot\cal C(c_s)$. In particular, $A^*_{\dot c}\nu(s)=0$ for all $s$.
\end{lem}
\begin{proof}
The first claim follows by observing that $\cal C(c_s)=\hat c(s)(\cal C(c))$.  Moreover, since $A_{\dot c_s}Z\in \cal C(c_s)$ for all $Z\in \cal H_{c(s)}$, $A^*_{\dot c_s}\nu(s)=0$.
\end{proof}

In a foliation with positive vertizontal curvature and bounded holonomy, we can prove that $\cal C(c)=\cal V_{c(0)}$.

\begin{theorem}\label{thm:AS}
Let $\cal F$  be a Riemannian foliation with bounded holonomy and positive vertizontal curvature. Then, for any horizontal geodesic $c:\bb R\to M$, $\cal C(c)=\cal V_{c(0)}$.
\end{theorem}
\begin{proof}
Let $\nu$ be a non-zero dual holonomy field as in lemma \ref{lem:Anu=0}. Using \eqref{eq:K}, we have 
%for $\varphi(t)=||\nu(t)||^2$:
\[\frac{1}{2}\frac{d^2}{dt^2}||\nu(t)||^2=K(\dot c,\nu(t))+3||S_{\dot c}\nu(t)||^2.\]
Since $K(\dot c,\nu(t))>0$,   $||\nu(t)||^2$ is unbounded, contradicting the hypothesis.
%Therefore $\varphi''>k\varphi>0$
%The strategy is to show that a dual holonomy field that starts orthogonal to $L^\#$ has unbounded norm. Suppose that there is a non-zero $\nu_0\in \cal V_p$ orthogonal to $T_p L_p^\#$. We show that the norm squared of the dual holonomy field with initial value $\nu_0$ is a smooth function that satisfies  $\varphi''>k\varphi$ for some $k>0$. In particular, $\varphi(t)>\alpha e^{\sqrt{k}t}+\beta$ or $\varphi(-t)>\alpha e^{\sqrt{k}t}+\beta$ for some constants $\alpha>0$ and $\beta$. 
\end{proof}

We believe that Theorem \ref{thm:AS} might help proving Wilhelm's conjecture.
\vspace{0.2cm}

Based on the nomenclature used for principal bundles, we call a foliation  \textit{irreducible }if it has only one dual leaf.

\subsection{Bounded holonomy as a generalization of compact holonomy}
Here, we present a characterization shared by both Riemannian submersions with compact holonomy groups and foliations with regular bounded holonomy.

\begin{theorem}\label{thm:2}
Let $\cal F$  be an irreducible Riemannian foliation with regular holonomy on a complete manifold $(M,g_0)$. Then  $\cal F$ has bounded holonomy if and only if $M$ admits a metric $g_1$ such that, for every $X\in \cal H$ and $\xi\in \cal V$,
\[g_1(X+\xi,X+\xi)=g_0(X,X)+g_1(\xi,\xi),\]
and all leaves are totally geodesic. 
\end{theorem}
\begin{proof}
Suppose that $\cal F$ has bounded holonomy. Then, the closure of $H_p$ on $\rm{Iso}(\cal V_p,\cal V_p)$ is compact and we can endow $\cal V_p$ with an $H_p$-invariant inner product $\lr{,}$. Let $\tau_p^*\cal V\to \cal E_p$ be the pull-back of $\pi:\cal V\to M$ and define the following bundle metric on $\tau_p^*\cal V$:
\[\lr{\xi,\eta}_h=\lr{h^{-1}\xi,h^{-1}\eta}.\]
It descends to a unique bundle metric $\hat g_1$ on $\cal V$: for $\xi,\eta\in\cal V_q$, and $h\in \tau_p^{-1}(q)$, let 
$\hat g_1(\xi,\eta)=\lr{\xi,\eta}_h$. To see that $\hat g_1$ is well defined, choose $k,h\in\tau_p^{-1}(q)$ and observe that
\[\lr{\xi,\eta}_h=\lr{h^{-1}\xi,h^{-1}\eta}=\lr{h^{-1}kk^{-1}\xi,h^{-1}kk^{-1}\eta}=\lr{\xi,\eta}_k,\]
where the last inequality follows since $h^{-1}k\in H_p$ and $\lr{,}$ is $H_p$-invariant. $\hat g_1$ is smooth, since it extends smoothly to $\tau_*\cal V\to \rm{Aut}(\cal V)$. Moreover, every element of $\cal E_p$ is tautologically an isometry with respect to $\hat g_1$: 
\[\lr{h\xi,h\eta}_{hk}=\lr{k^{-1}h^{-1}h\xi,k^{-1}h^{-1}h\eta}=\lr{\xi,\eta}_k.\]
Recalling Proposition \ref{prop:hol}, item (2), we conclude that any pair of holonomy fields have constant inner product in the metric $g_1(X+\xi,X+\xi)=g_0(X,X)+\hat g_1(\xi,\xi)$. This proves that the $S$-tensor vanishes in $g_1$.

The converse is also valid: if $M$ admits a metric where all leaves are totally geodesic, all $h\in \cal E$ are isometries. Therefore, the collection of elements in all $H_p$ have uniformly bounded norms,  since each $H_p$ is bounded and $H_q=hH_ph^{-1}$ for any $h$ such that $h(p)=q$.
\end{proof}

Although Theorem \ref{thm:2} seems quite strong, it may be the case that bounded holonomy only occurs with regular holonomy. Such behavior is present in the context of Riemannian submersions, where compact holonomy implies a smooth structure on $\cal E$ (from the proof of \cite[Theorem 2.7.2]{gw}).
	
Theorem \ref{thm:2} also suggests a possible interplay between bounded holonomy, regularity of the holonomy and smoothness of $H_p$. We close this section conjecturing how the relations between these structures might be.

\begin{conj}
A Riemannian foliation has regular holonomy if and only if $H_p$ inherits a Lie structure from the inclusion $H_p\subset \rm{Iso}(\cal V_p,\cal V_p)$.
\end{conj}
\begin{conj}
If a Riemannian foliation has bounded holonomy (and no leaf has infinite fundamental group), then $H_p$ is compact and the holonomy is regular.
\end{conj}

\bibliographystyle{alpha}
\bibliography{bib1}

\begin{thebibliography}{GW09}

\bibitem[DS]{DS-A-flat}
C.~Duran and L.~Speran\c{c}a.
\newblock Rigidity of flat sections on non-negatively curved pullback
  submersions.
\newblock {\em Manuscripta Math.}, 147:511--525.

\bibitem[GG]{gg}
D.~Gromoll and K.~Grove.
\newblock The low-dimensional metric foliations of euclidean spheres.
\newblock {\em J. Diff. Geo.}, 28(1):143--156.

\bibitem[GW09]{gw}
D.~Gromoll and G.~Walshap.
\newblock {\em Metric Foliations and Curvature}.
\newblock Birkhäuser Verlag, Basel, 2009.

\bibitem[KN]{knI}
S.~Kkobayashi and K.~Nomizu.
\newblock {\em Foundations of differential geometry}, volume~I.
\newblock Interscience Publishers.

\bibitem[Pet06]{p}
P.~Petersen.
\newblock {\em Riemannian Geometry}.
\newblock Springer, 2006.

\bibitem[PW]{pro-flats}
C.~Pro and F.~Wilhelm.
\newblock Flats and submersions in nonnegative curvature.
\newblock {\em Geometriae Dedicata}, 161:109--118.

\bibitem[SS]{ss-hol}
V.~Schroeder and M.~Strake.
\newblock Volume growth of open manifolds with nonnegative curvature.
\newblock 8:159--165.

\bibitem[Tapa]{tapp-bounded-sub}
K.~Tapp.
\newblock Bounded riemannian submersions.
\newblock {\em Indiana University Mathematics Journal,}, 49:637--354.

\bibitem[Tapb]{tapp-vol-grwoth}
K.~Tapp.
\newblock Volume growth and holonomy in nonnegative curvature.
\newblock {\em Proc. Amer. Math. Soc.}, 127:3035--3041.

\bibitem[VZ]{verdianiziller}
L.~Verdiani and W.~Ziller.
\newblock Concavity and rigidity in non-negative curvature.
\newblock {\em J. Diff. Geom.}, 97:349--375.

\bibitem[Wil]{wilkilng-dual}
B.~Wilking.
\newblock A duality theorem for riemannian foliations in nonnegative sectional
  curvature.
\newblock {\em Geom. Func. Anal.}, 17:1297--1320.

\end{thebibliography}

\end{document}